\documentclass[a4paper,11pt]{article}
\usepackage{stmaryrd}
\usepackage{amsfonts}
\usepackage{bbm}

\usepackage{amscd}
\usepackage{mathrsfs}
\usepackage{latexsym,amssymb,amsmath,amscd,amscd,amsthm,amsxtra,xypic}
\usepackage[dvips]{graphicx}
\usepackage[utf8]{inputenc}
\usepackage[T1]{fontenc}
\usepackage{enumerate}
\usepackage{lmodern}
\usepackage{amssymb}
\usepackage[all]{xy}
\usepackage{nicefrac,mathtools,enumitem}
\usepackage{microtype}
\usepackage{amsfonts}
\usepackage{amssymb}
\usepackage{mathrsfs}
\usepackage{amsmath}
\usepackage{amsthm}
\usepackage{enumerate}
\usepackage{indentfirst}
\usepackage{amsfonts}
\usepackage{amssymb}
\usepackage{mathrsfs}
\usepackage{amsmath}
\usepackage{amsthm}
\usepackage{enumerate}
\usepackage{cite}
\usepackage{mathrsfs}
\usepackage{kpfonts}
\usepackage{geometry}
\allowdisplaybreaks

\newtheorem{thm}{Theorem}[section]
\newtheorem{lem}[thm]{Lemma}
\newtheorem{cor}[thm]{Corollary}
\newtheorem{pro}[thm]{Proposition}
\newtheorem{rmk}[thm]{Remark}

\numberwithin{equation}{section}

\theoremstyle{definition}
    \newtheorem{defi}[thm]{Definition}

    \newtheorem{ex}[thm]{Example}

\newcommand{\gl}{\mathfrak {gl}}

\newcommand{\ad}{\mathsf{ad}}

\def\c{\cdot}

\def\o{\otimes}

\begin{document}

\title{Extensions and crossed modules of $n$-Lie Rinehart algebras}

\author{\normalsize \bf  A. Ben Hassine \small{$^{1,2}$}\footnote {  E-mail: benhassine.abdelkader@yahoo.fr}
, T. Chtioui\small{$^{2}$}\footnote {  E-mail: chtioui.taoufik@yahoo.fr}
, M.  Elhamdadi\small{$^{3}$}\footnote { Corresponding author,  E-mail: emohamed@math.usf.edu}
 \ \ \  and\ \ \   S. Mabrouk\small{$^{4}$}\footnote { E-mail: mabrouksami00@yahoo.fr}}
\date{{\small{$^{1}$ Department of Mathematics, College  of Science and Arts at
Belqarn, P. O. Box 60, Sabt Al-Alaya 61985, University of Bisha, Saudi Arabia \\  \small{$^{2}$    Faculty of Sciences, University of Sfax,   BP
1171, 3000 Sfax, Tunisia \\ \small{$^{3}$} Department of Mathematics,
University of South Florida, Tampa, FL 33620, U.S.A.\\ \small{$^{4}$} Faculty of Sciences, University of Gafsa,   BP
2100, Gafsa, Tunisia
 }}}} \maketitle
\date{}
\date{}

\maketitle

\abstract{
We introduce a notion of $n$-Lie Rinehart algebras as a generalization of Lie Rinehart algebras to $n$-ary case. This notion is  also an algebraic analogue   of   $n$-Lie algebroids. We develop  representation theory and  describe a cohomology complex  of  $n$-Lie Rinehart algebras. Furthermore, we investigate extension theory of $n$-Lie Rinehart algebras by means of $2$-cocycles.  Finally, we introduce crossed modules of $n$-Lie Rinehart algebras to gain a better understanding of their third dimensional cohomology groups.
}

\

\tableofcontents
\section{Introduction}
In 1985, Filippov \cite{Filippov:nLie} introduced the concept of $n$-Lie algebras and classified the $(n+1)$-
dimensional $n$-Lie algebras over an algebraically closed field of characteristic zero. The
structure of $n$-Lie algebras is very different from that of Lie algebras due to the $n$-ary
multilinear operations involved. On the other hand, in 1973,  Nambu introduced an $n$-ary generalization of
Hamiltonian dynamics by means of $n$-ary Poisson bracket (\cite{Nambu:GenHD}). Apparently, Nambu was motivated by some problems of quark dynamics. In  \cite{Takhtajan}, Takhtajan  developed the foundations of the
theory of $n$-Poisson or Nambu–Poisson manifolds.  Over the years, many of the structural results in the theory of Lie algebras have  
 been generalized to the context of n-Lie algebras, although it seems that the number of $n$-Lie algebras
decreases as $n$ increases and this is due to the fact that the $n$-Jacobi
identity imposes 
strong conditions as $n$ increases. For more results 
on n-Lie algebras the reader  can consult 
 \cite{Ho&Chebotar&Ke,Papadopoulos,Alekseevsky&Guha,Gautheron,Michor&Vinogradov,Marmo&Vilasi&Vinogradov,Nakanishi}.  For the
construction, realization and classifications of $3$-Lie algebras and $n$-Lie algebras see  \cite{Arnlind&Makhlouf&Silvestrov,Bai&Bai&Wang,Bai&Song&Zhang} . In particular, representation theory of $n$-Lie algebras was first introduced by Kasymov in \cite{Kasymov}. The
adjoint representation is defined by the ternary bracket in which two elements are fixed.
To every Numbu algebra, a Leibiz algebra was canonically associated in \cite{Dal_Takh}.  Thus 
one may represent a $3$-Lie algebra and more generally
an $n$-Lie algebra by a Leibniz algebra \cite{Dal_Takh}. Following this approach, deformations of $3$-Lie
algebras and $n$-Lie algebras were studied in \cite{Figueroa,Makhlouf2016,Takhtajan1995}.  The author of \cite{Mikolaj}
defined a graded Lie algebra structure on the cochain complex of a $n$-Leibniz algebra and
described a $n$-Leibniz structure as a canonical structure.

Extension theory is an important tool of constructing a larger algebra from a given one, and there has been
many types of extensions such as double extensions and Kac–Moody extensions. In
1997, Bordemann \cite{Bordemann} introduced a notion of $T^\ast$-extensions of Lie algebras and
proved that every nilpotent finite-dimensional algebra over an algebraically closed
field carrying a nondegenerate invariant symmetric bilinear form is a suitable $T^\ast$-
extension. The method of $T^\ast$-extension was used in \cite{Bajo&Benayadi&Medina,Medina&Revoy} and generalized to other types of algebras recently in \cite{Bai&Li,Lin&Wang&Deng,Liu&Zhang,Liu&Zhang1}.


Crossed modules for many algebraic structures, such as groups, Lie algebras, Hom-Lie algebras, Leibniz $n$-algebra  and Lie-Rinehart algebras, are classified by the third dimensional cohomology group.  For more details see  (\cite{Casas&Ladra&Pirashvili,Casas&Khmaladze&Ladra,Casas&Xabier,Das}).

The notion of Lie Rinehart algebra plays an important role in many branches of mathematics. The idea of this notion goes back to the work of Jacobson on studying certain field
extensions. It also appeared with different names in several areas of mathematics such as differential geometry and differential Galois theory, see for example \cite{Higgins&Mackenzie}.   
Huebschmann viewed Lie Rinehart algebras
as algebraic counterpart of Lie algebroids defined over smoothmanifolds. His work
on several aspects of these algebras has been developed systematically through a series of
articles namely \cite{Huebschmann3,Huebschmann,Huebschmann1}. Lie Rinehart structures have been the subject of extensive studies,
in relations to symplectic geometry, Poisson structures, Lie groupoids and algebroids and
other type of quantizations \cite{Herz,Liu&Sheng&Bai&Chen,Liu&Sheng&Bai,Liu&Sheng&Bai1,Sheng&Zhu,Sheng2012,Mackenzie1987,Mackenzie1995,Mishra&Mukherjee&Naolekar}. Lie Rinehart algebras have been investigated further in many papers \cite{Hassine&Chtioui&Elhamdadi&Mabrouk,Bkouche,Casas,Casas&Ladra&Pirashvili,Casas&Ladra&Pirashvili1,Chen&Liu&Zhong,Huebschmann1,Krahmer&Rovi}. Some generalizations of Lie Rinehart algebras, such as Lie Rinehart superalgebras \cite{Chemla} or restricted Lie Rinehart algebras \cite{Dokas} or $3$-Lie Rinehart superalgebras \cite{BenHassine&Chtioui&Mabrouk&Silvestrov}  have been recently studied.

This paper is organized as follows:  In Section~\ref{AKMS:n-Lie}, we recall some basic definitions of $n$-Lie algebras and $n$-Lie algebroids.   Section~\ref{n-LieR} is devoted to introducing the notion of $n$-Lie Rinehart algebras which is a generalization of Lie Rinehart algebras and giving some construction results. In Section~\ref{representation-n-Lie} we investigate the notions of a representation of $n$-Lie Rinehart algebras and a cohomology of $n$-Lie Rinehart algebras with coefficients in a module.  
Section~\ref{Extensions} deals with extension theory of $n$-Lie Rinehart algebras by means of $2$-cocycles. Finally, Section \ref{Crossed} defines a notion of crossed modules for  $n$-Lie Rinehart algebras and gives a relation to the third dimensional cohomology group.

In this paper, all vector spaces are over a field $\mathbb{K}$ of characteristic $0$.

\section{Preliminaries}\label{AKMS:n-Lie}

In this section, we recall some basic definitions and results of $n$-Lie algebras based on \cite{Filippov:nLie,Kasymov}. Moreover we give a (modified) definition of $n$-Lie algebroids which is already given in \cite{GM}.

\begin{defi}
   An $n$-Lie algebra is a vector space $L$ equipped with a bracket operation $[\cdot,\cdots,\cdot]:\wedge^n L\longrightarrow L $  such that for all $x_1,\cdots,x_{n-1}$, $y_1,\cdots, y_n\in L $, we have :
 \begin{align}\label{Fundamental.identity}
[x_1,\cdots,x_{n-1},[y_1,\cdots,y_n] ] &=\sum_{i=1}^n[y_1,\cdots,y_{i-1},
     [x_1,\cdots,x_{n-1},y_i] ,y_{i+1},\cdots,y_n].
\end{align}
 \end{defi}

 The elements in $\wedge^{n-1}L$ are called fundamental elements.
 On $\wedge^{n-1}L$, one can define a new bracket operation $[\cdot,\cdot]_F$ by
\begin{equation}\label{bracket_F}
  [X,Y]_F=\sum_{i=1}^{n-1}y_1\wedge\cdots\wedge y_{i-1}\wedge[x_1,\cdots,x_{n-1},y_i]_L \wedge y_{i+1}
  \wedge\cdots\wedge y_{n-1},
\end{equation}
for all $X=x_1\wedge\cdots \wedge x_{n-1}$ and $Y=y_1\wedge\cdots\wedge y_{n-1}$.
It is proved in \cite{Takhtajan:cohomology} that $(\wedge^{n-1}L ,[\cdot,\cdot]_F)$ is a Leibniz algebra.

\begin{defi}
  A morphism  of $n$-Lie algebras $f:(L ,[\cdot,\cdots,\cdot]_{L })\longrightarrow
   (L',[\cdot,\cdots,\cdot]_{L'})$ is a linear map $f:L \longrightarrow L'$ such that
  \begin{eqnarray}
f[x_1,\cdots,x_n]_{L }&=&[f(x_1),\cdots,f(x_n)]_{L'},\hspace{3mm}\forall x_1,\cdots,x_n\in L.   \end{eqnarray}
\end{defi}

\begin{defi}\label{defi:rep}
  A  representation  of an $n$-Lie algebra $(L,[\cdot,\cdots,\cdot] )$ on a vector space $M$  is a linear map $\rho:\wedge^{n-1}L \longrightarrow\gl(M)$
   such that for all  $x_1,\cdots,x_{n-1},y_1,\cdots,y_n\in L$, we have
 \begin{align}
& \rho(x_1,\cdots,x_{n-1})\circ\rho(y_1,\cdots,y_{n-1})-\rho(y_1,\cdots,y_{n-1})\circ \rho(x_1,\cdots,x_{n-1})\nonumber\\ \label{rep1}
=&\displaystyle\sum_{i=1}^{n-1}\rho(y_1,y_2,\cdots,[x_1,\cdots,x_{n-1},y_i],y_{i+1},\cdots,y_{n-1});\\
        & \rho(x_1,\cdots,x_{n-2},[y_1,\cdots,y_n])\nonumber\\ \label{rep2}
         =&\sum_{i=1}^n(-1)^{n-i}\rho(y_1,\cdots,\widehat{y_i},\cdots,y_n)\circ\rho(x_1,\cdots,x_{n-2},y_i).
\end{align}
\end{defi}
We denote such a representation by $(M,\rho)$.

\begin{ex}
Define $\ad:\wedge^{n-1}L\longrightarrow\gl(L)$ by
\begin{equation}
  \ad(X)y=[x_1,\cdots,x_{n-1},y],\quad\forall X=x_1\wedge\cdots \wedge x_{n-1}\in\wedge^{n-1}L,~y\in L.
\end{equation}
Then $(L,\ad)$ is a representation of the $n$-Lie algebra $(L,[\cdot,\cdots,\cdot])$ on $L$,
 which is called the adjoint representation.

\end{ex}

\begin{pro}\label{tri}
Let  $(M,\rho)$ be a representation of an $n$-Lie algebra $(L,[\cdot,\cdots,\cdot])$.
Define a bracket operation $[\cdot,\cdots,\cdot]_{\rho}:\wedge^n(L\oplus M)\longrightarrow L\oplus M$ by
$$[x_1+m_1,\ldots,x_n+m_n]_{\rho}=[x_1,\ldots,x_n]+\sum_{i=1}^{n}(-1)^{n-i}\rho(x_1,\ldots,\widehat{x}_i,\ldots,x_n)m_i,\,\,\,\,\forall x_i\in L,m_i\in M.$$
Then $(L\oplus M,[\cdot,\cdots,\cdot]_{\rho})$ is an $n$-Lie algebra, which we call the { semi-direct product} of the  $n$-Lie algebra
$(L,[\cdot,\cdots,\cdot])$ by the representation $(M,\rho)$. We denote this semi-direct $n$-Lie algebra simply by $L\ltimes M$.
\end{pro}

 The notion of Filippov algebroid called also $n$-Lie algebroid is introduced in \cite{GM} as a generalization of Lie algebroids.  We recall the definition.

\begin{defi}
An $n$-Lie algebroid  is a vector bundle $\pi:A\to M$ endowed with an  $n$-Lie  algebra structure  on the space $\Gamma A$ of smooth sections of $A$ given by the bracket
$$[~, \ldots, ~]: \underbrace{\Gamma A \times {\cdots} \times \Gamma A}_{n} \rightarrow \Gamma A,$$
 and a bundle map  $\rho: \wedge ^{n-1}A \rightarrow TM$ (the anchor map) satisfying the following conditions:\\
(a) For all $x_1, \ldots, x_{n-1}, y_1, \ldots, y_{n-1} \in \Gamma A,$
\begin{equation}\label{algebroid1}
    [\rho(x_1, \cdots , x_{n-1}), \rho(y_1 , \cdots , y_{n-1})]= \sum_{i=1}^{n-1} \rho(y_1, \cdots ,[x_1, \ldots, x_{n-1}, y_i], y_{i+1} , \cdots, y_{n-1}),
\end{equation}
\begin{equation}\label{algebroid2}
    \rho(x_1,\cdots,x_{n-2},[y_1,\cdots,y_n])
         =\sum_{i=1}^n(-1)^{n-i}\rho(y_1,\cdots,\widehat{y_i},\cdots,y_n)\circ\rho(x_1,\cdots,x_{n-2},y_i)
\end{equation}
where the bracket on the left hand side is the usual Lie bracket on vector fields.\\
(b) For all $x_1, \ldots, x_{n-1}, y\in \Gamma A$ and $f \in C^\infty(M)$,\\
$$[x_1, \ldots, x_{n-1}, fy]= f[x_1, \ldots, x_{n-1}, y] + \rho(x_1,\cdots, x_{n-1})(f)y.$$
\end{defi}
Note that in \cite{GM}, the condition \eqref{algebroid2} does not exist. In fact, they consider 
a different definition of the representation of an $n$-Lie algebra.

In the following section, we introduce an algebraic analogue of Filippov algebroid called $n$-Lie Rinehart algebra.

\section{$n$-Lie Rinehart algebras}\label{n-LieR}
In this section, we introduce a notion of $n$-Lie Rinehart algebras which is a generalization of Lie Rinehart algebras  given in \cite{Chemla}. We recall the notion of Leibniz-Rinehart algebra and generalize the construction of Taktajan to $n$-Lie Rinehart algebra.

\begin{defi}\label{def-3Lie rinh super}
A $n$-Lie Rinehart algebra over $A$ is a tuple $(L, A, [\cdot,\cdots,\cdot],\rho)$,
where $A$ is an associative commutative algebra, $L$ is an $A-$module, $[\cdot,\cdots, \cdot]: L\times \cdots \times L \to L$ is
a  skew-symmetric $n$-linear map, and a linear map $\rho:\wedge^{n-1} L \to Der(A)$
such that the following conditions hold:
\begin{enumerate}
\item $(L, [\cdot,\cdots,\cdot])$ is an $n$-Lie algebra.
\item The map $\rho$ is a representation of $(L, [\cdot,\cdots,\cdot])$ on $A$.
\item  For all $x_1,\cdots, x_{n-1}\in L, a\in A,$
\begin{equation}\label{eq:Rinhart2}
\rho(ax_1,\cdots,x_{n-1})=a\rho(x_1,\cdots,x_{n-1}).
\end{equation}
\item The compatibility condition:
\begin{equation}\label{eq:Rinhart1}
    [x_1, \cdots, ax_n]=a[x_1, \cdots, x_n]+\rho(x_1,\cdots,x_{n-1})a x_n,
\end{equation}
for all $x_1,\cdots,x_n\in L, a\in A$.
\end{enumerate}
\end{defi}
\begin{ex}
 Let $(A, [~, \ldots, ~], \rho)$ be an $n$-Lie algebroid over a smooth manifold $M$.  This $n$-Lie algebroid provides an $n$-Lie Rinehart algebra $(\Gamma A , C^\infty(M),[\cdot,\cdots,\cdot],\rho)$.
\end{ex}

\begin{rmk}
	We have the following observations:
\begin{enumerate}
\item If $\rho=0$, then  $(L, A,[\cdot,\cdots,\cdot])$ is called  an $n$-Lie $A$-algebra.
\item If $A=\mathbb K$, then  $(L, \mathbb K,[\cdot,\cdots,\cdot])$ is an  $n$-Lie algebra.
\item If $A=L$, then  $(L, L,[\cdot,\cdots,\cdot])$ is an $n$-ary Poisson algebra.
\end{enumerate}
\end{rmk}

In the following definition, we recall the notion of Leibniz Rinehart algebra as a particular case of Hom-Leibniz Rinehart algebra introduced in \cite{Guo&Zhang&Wang}.

\begin{defi}
A Leibniz-Rinehart algebra over $A$ is a tuple $(L, A,  [\c, \c], \rho)$, where $A$ is an
associative commutative algebra, $L$ is an $A$-module, $[\c, \c] : L\times L \rightarrow L$ is a  linear map
and the linear map $\rho: L\rightarrow Der(A)$  satisfying the  following conditions.
\begin{enumerate}
    \item The pair  $(L,  [\c, \c])$ is a Leibniz algebra.
\item $\rho([x,y])=\rho(x)\rho(y)-\rho(y)\rho(x)$.
\item $\rho (ax)=a \rho(x)$.
\item Compatibility condition
\begin{eqnarray}\label{Compatibilty-Leibniz}
[x, a y] &= a [x, y] + \rho(x)(a)y.
\end{eqnarray} for all $a\in A, x, y\in L$.
\end{enumerate}
\end{defi} Let $(L, A,[\cdot,\cdots,\cdot], \rho)$   be a $n$-Lie Rinehart algebra and denote $\wedge^{n-1}L$ by $\mathcal{L}$.  Define the linear map $\widehat{\rho}:\mathcal{L}\to Der(A)$ by:
$$
\widehat{\rho}(X)=\rho(x_1,\cdots,x_{n-1}),\ \text{for~any} ~X=x_1\wedge\cdots\wedge x_{n-1}.
$$
The space $\mathcal{L}$ becomes an $A$-module via the following
action $$aX=x_1\wedge \cdots\wedge ax_i\wedge \cdots\wedge x_{n-1},\ \text{for~any} ~X=x_1\wedge\cdots\wedge x_{n-1},\ i=1,...,n-1.$$
\begin{pro}
  With the above notations, $(\mathcal L, A,[\cdot,\cdot]_F,\widehat{\rho})$ is a Leibniz-Rinehart algebra, where the bracket
is defined as in Eq. \eqref{bracket_F}.
\end{pro}

\begin{proof}
For all $X,Y\in \mathcal L$ and $a\in A$, using Eq.\eqref{rep1}, we have
$$\widehat{\rho}([X,Y]_F)=\widehat{\rho}(X)\widehat{\rho}(Y)-\widehat{\rho}(Y)\widehat{\rho}(X), \forall X,Y\in \mathcal L.$$
Furthermore,  we have
$$
\widehat{\rho} (aX)= \rho(ax_1,\cdots,x_{n-1})=a \rho(x_1,\cdots,x_{n-1})=a\widehat{\rho} (X).
$$
Thanks to Eqs.\eqref{bracket_F} and \eqref{eq:Rinhart1} we have
{\small\begin{align*}
    [X,aY]_F
    =& [x_1,\cdots,x_{n-1},ay_1]\wedge y_2\wedge\cdots\wedge y_{n-1}+\sum_{i=2}^{n-1}ay_1\wedge\cdots\wedge y_{i-1}\wedge[x_1,\cdots,x_{n-1},y_i] \wedge y_{i+1}\\
    =& a[x_1,\cdots,x_{n-1},y_1]\wedge y_2\wedge\cdots\wedge y_{n-1}+\rho(x_1,\cdots,x_{n-1})ay_1\wedge y_2\wedge\cdots\wedge y_{n-1}\\
    &+a\sum_{i=2}^{n-1}y_1\wedge\cdots\wedge y_{i-1}\wedge[x_1,\cdots,x_{n-1},y_i] \wedge y_{i+1}
  \wedge\cdots\wedge y_{n-1}\\
  =&a\Big([x_1,\cdots,x_{n-1},y_1]\wedge y_2\wedge\cdots\wedge y_{n-1}+\sum_{i=2}^{n-1}y_1\wedge\cdots\wedge y_{i-1}\wedge[x_1,\cdots,x_{n-1},y_i]
  \wedge\cdots\wedge y_{n-1}\Big)\\
  &+\rho(x_1,\cdots,x_{n-1})a(y_1\wedge y_2\wedge\cdots\wedge y_{n-1})\\
  =&a[X,Y]_F+\widehat{\rho}(X)aY.
\end{align*}}
Therefore, $(\mathcal L, A,[\cdot,\cdot]_F,\widehat{\rho})$ is a Leibniz-Rinehart algebra.
\end{proof}

\begin{rmk}
If the condition of Eq. \eqref{eq:Rinhart2} is not satisfied then we call $(L, A, [\cdot,\cdots,\cdot],\rho)$ a weak $n$-Lie Rinehart algebra over $A$.
\end{rmk}

\begin{defi}\label{homorphism-Rinehart}
Let $(L, A,[\cdot,\cdots,\cdot]_L, \rho)$ and $(L', A',[\cdot,\cdots,\cdot]_{L'}, \rho')$ be two $n$-Lie Rinehart algebras.
  Then an $n$-Lie Rinehart algebra homomorphism is defined as a pair of maps $(g, f)$, where
the maps $g: A \rightarrow A'$ and $f:L\rightarrow L'$ are two $\mathbb K$-algebra homomorphisms such that:
\begin{itemize}
\item[(1)] $f(ax) = g(a)f(x)$ for all $x \in L, a \in A,$
\item[(2)] $g(\rho(x_1,\dots,x_{n-1})(a)) = \rho'(f(x_1),\cdots,f(x_{n-1}))(g(a))$ for all $x_i\in L, a \in A.$
\end{itemize}

\end{defi}

The following definition introduces the notions of subalgebra, ideal and indecomposability in the context of $n$-Lie Rinehart algebras. 

\begin{defi}\label{defn:subandideal} Let $(L,A,[\cdot,\cdots,\cdot],\rho)$ be a $n$-Lie Rinehart algebra.
\begin{enumerate}
 \item If $S$ is a subalgebra of the $n$-Lie algebra $(L,A,[\cdot,\cdots,\cdot],\rho)$ satisfying $AS\subset S$,
 then   $(S, A,[\cdot,\cdots,\cdot],\rho|_{\wedge^{n-1}S})$ is an $n$-Lie Rinehart algebra, which is called a subalgebra of the $n$-Lie Rinehart algebra
 $(L,A,[\cdot,\cdots,\cdot],\rho)$.
\item If $I$ is an ideal of the $n$-Lie   algebra $(L,A,[\cdot,\cdots,\cdot],\rho)$ and satisfies
 $AI\subset I$ and $\rho(I,L,\cdots,L)(A)L\subset I$, then   $(I, A,[\cdot,\cdots,\cdot],\rho|_{\wedge^{n-1}I})$ is an $n$-Lie Rinehart algebra,
 which is called { an ideal of the $n$-Lie Rinehart algebra $(L,A,[\cdot,\cdots,\cdot],\rho)$.}
\item If a $n$-Lie Rinehart algebra $(L,A,[\cdot,\cdots,\cdot],\rho)$  cannot be decomposed into the direct sum of two nonzero ideals,
then $L$ is called an indecomposable $n$-Lie Rinehart algebra.
\end{enumerate}
\end{defi}

\begin{thm}\label{Tensor}
Let   $(L,A, [\cdot, \cdots, \cdot], \rho)$ be an $n$-Lie Rinehart algebra. Then $(A\otimes L,A,[\cdot,\cdots,\cdot]_{A\otimes L},\rho_{A\otimes L})$ is a $n$-Lie Rinehart algebra, where
\begin{align}
&[a_1\otimes x_1, \dots, a_n\otimes x_n]_{A\otimes L}= a_1\dots a_n\otimes [x_1, \dots, x_n],\label{bracket1}\\
&\rho_{A\otimes L}(a_1\otimes x_1, \dots, a_{n-1}\otimes x_{n-1})= a_1\cdots a_{n-1} \rho(x_1, \dots, x_{n-1})\label{rho},
\end{align}
for all $x_1,\cdots, x_{n}\in L$ and $a_1,\cdots, a_n\in A$. Where $A$ acts on $A\otimes L$ via
$$
b(a_1\otimes x_1)=ba_1\otimes x_1=a_1\otimes bx_1, \forall b,a_1\in A,~~x_1\in L.
$$
\end{thm}
\begin{proof}
It is obvious to check that $(A\otimes L,[\cdot,\cdots,\cdot]_{A\otimes L})$ is an $n$- Lie algebras, $\rho_{A\otimes L}$ is a representation of $L$ and $\rho_{A\otimes L}(a_1\otimes x_1, \dots, a_{n-1}\otimes x_{n-1})\in Der(A)$. For any $b,a_1,\cdots, a_{n-1},a_n\in A$ and $x_1,\cdots,x_{n-1},x_n\in L$, we have
\begin{align*}
\rho_{A\otimes L}(b(a_1\otimes x_1), \dots, a_{n-1}\otimes x_{n-1})=&\rho_{A\otimes L}(ba_1\otimes x_1, \cdots, a_{n-1}\otimes x_{n-1})\\
=&(ba_1)\cdots a_{n-1}\rho(x_1,\cdots,x_{n-1})\\
=&ba_1\cdots a_{n-1}  \rho(x_1,\cdots,x_{n-1})\\
=&b\big(a_1\cdots a_{n-1}  \rho(x_1,\cdots,x_{n-1})\big)\\
=&b\rho_{A\otimes L}(a_1\otimes x_1, \dots, a_{n-1}\otimes x_{n-1}).
\end{align*}
To prove the compatibility condition \eqref{eq:Rinhart1}, we compute the following
  \begin{align*}
&[a_1\otimes x_1, \dots, b(a_n\otimes x_n)]_{A\otimes L}\\
=&[a_1\otimes x_1, \dots, a_n\otimes bx_n]_{A\otimes L}\\
=& a_1\cdots a_n\otimes [x_1, \dots, bx_n]\\
=&a_1\cdots a_n\otimes (b[x_1, \dots, x_n])+a_1\dots a_n\otimes \rho(x_1, \cdots, x_{n-1})bx_n\\
=&b[a_1\otimes x_1, \dots, a_n\otimes x_n]_{A\otimes L}+ a_n\otimes \rho_{A\otimes L}(a_1\otimes x_1, \dots, a_{n-1}\otimes x_{n-1})bx_n\\
=&b[a_1\otimes x_1, \dots, a_n\otimes x_n]_{A\otimes L}+ \rho_{A\otimes L}(a_1\otimes x_1, \dots, a_{n-1}\otimes x_{n-1})b(a_n\otimes x_n).
\end{align*}
Therefore, $(A\otimes L,A,[\cdot,\cdots,\cdot]_{A\otimes L},\rho_{A\otimes L})$ is a $n$-Lie Rinehart algebra.
\end{proof}
\begin{thm}
Let $(L, A,[\cdot,\cdots,\cdot],\rho)$ be an $n$-Lie Rinehart algebra and
$$E=L\oplus A=\{(x,a)\  \ |\ \ x\in L, a\in A\}.$$ Then
$(E, A, [\cdot,\cdots,\cdot]',\rho')$ is an $n$-Lie Rinehart algebra, where,
\begin{equation}\label{eq16}
 b(x_i,a_i) = (bx_i,ba_i),
 \end{equation}
 \begin{equation}\label{eq17}
 [(x_1, a_1), \cdots, (x_n, a_n)]' = \big([x_1,\cdots, x_n],\displaystyle\sum_{i=1}^n (-1)^{n-i}\rho(x_1, \cdots,\widehat{x_i},\cdots,x_n)a_i \big),
 \end{equation}
 \begin{equation}\label{eq18}
\rho' : \wedge^{n-1} E \rightarrow Der(A), \rho'\big((x_1, a_1),\cdots,(x_{n-1}, a_{n-1})\big) = \rho(x_1,\cdots, x_{n-1})
\end{equation}
for any $b, a_i\in A, x_i\in L,~~  i=1,\dots,n$.
\end{thm}
\begin{proof}
Let  $x_i \in L, \ a_i \in A,\ i=1,\dots, n$. According to Proposition \ref{tri},  $(E,[\cdot,\dots,\cdot]')$ is a $n$-Lie algebra. Thanks to Eq \eqref{eq16}, $E$ is an $A$-module.
Since $\rho$ is a representation of $L$ on $A$ with values in $Der(A)$ then  $\rho'$ is a representation of $E$ on $A$ with values in $Der(A)$. Indeed, for any $b_1,b_2\in A$, we have
\begin{align*}
&\rho'\big((x_1, a_1),\cdots,(x_{n-1}, a_{n-1})\big)b_1b_2\\
&=\rho(x_1,\cdots,x_1)\big)b_1b_2\\
&=(\rho(x_1,\cdots,x_{n-1})b_1)b_2+b_1\rho(x_1,\cdots,x_1)\big)b_2\\
&=\Big(\rho'\big((x_1, a_1),\cdots,(x_{n-1}, a_{n-1})\big)b_1\Big)b_2+b_1\rho'\big((x_1, a_1),\cdots,(x_{n-1}, a_{n-1})\big)b_2.
\end{align*}
Then, $\rho'\big((x_1, a_1),\cdots,(x_{n-1}, a_{n-1})\big)\in Der(A)$. Moreover,
\begin{align*}
\rho'\big(b_1(x_1, a_1),\cdots,(x_{n-1}, a_{n-1})\big)&=\rho'\big((b_1x_1, b_1a_1),\cdots,(x_{n-1}, a_{n-1})\big)\\
&=\rho(b_1x_1,\cdots,(x_{n-1})\\
&=b_1\rho(x_1,\cdots,x_{n-1})\\
&=b_1\rho'\big((x_1, a_1),\cdots,(x_{n-1}, a_{n-1})\big).
\end{align*}
Then,  Eq.\eqref{eq:Rinhart2} holds.
To prove the compatibility condition \eqref{eq:Rinhart1}, we have
\begin{align*}
&[(a_1,x_1), \cdots, b_1(a_n, x_n)]'\\
=&[(a_1,x_1), \cdots, (b_1a_n, b_1x_n)]'\\
=&\big([x_1,\cdots, b_1x_n],\displaystyle\sum_{i=1}^{n-1} (-1)^{n-i}\rho(x_1, \cdots,\widehat{x_i},\cdots,b_1x_n)a_i+(-1)^{n-1}\rho(x_1,\cdots,x_{n-1})b_1a_n \big)\\
=&\big([x_1,\cdots, x_n]+\rho(x_1,\cdots,x_{n-1})b_1x_n,\displaystyle\sum_{i=1}^{n-1} (-1)^{n-i}b_1\rho(x_1, \cdots,\widehat{x_i},\cdots,x_n)a_i\\
&+b_1\rho(x_1, \cdots,x_{n-1})a_n +(\rho(x_1, \cdots,x_{n-1})b_1)a_n\big)\\
=&\big([x_1,\cdots, x_n],\displaystyle\sum_{i=1}^{n} (-1)^{n-i}b_1\rho(x_1, \cdots,\widehat{x_i},\cdots,x_n)a_i\big)\\
&+\big(\rho(x_1,\cdots,x_{n-1})b_1x_n,(\rho(x_1,\cdots,x_{n-1})b)a_n\big)\\
=&b_1[(x_1, a_1), \cdots, (x_n, a_n)]' +\rho'\big((x_1, a_1),\cdots,(x_{n-1}, a_{n-1})\big)b(x_n,a_n).
\end{align*}
Then, $(E,A,[\cdot,\cdots,\cdot]',\rho')$ is $n$-Lie Rinehart algebras.
\end{proof}
\section{Representations of $n$-Lie Rinehart algebras}\label{representation-n-Lie}

In this section, we study representation theory and cohomology of $n$-Lie Rinehart algebras. Let   $(L, A, [\cdot,\cdots,\cdot],\rho)$ be an $n$-Lie Rinehart algebra.

\begin{defi} \label{def-rep} Let $M$ be an $A$-module and  $\psi: \wedge^{n-1} L\rightarrow End(M)$ be a linear map. The pair $(M,\psi)$ is called a representation of  $(L, A, [\cdot,\cdots,\cdot],\rho)$ if the following conditions hold:
\begin{enumerate}
\item  $\psi$ is a representation of $(L, [\c, \cdots, \c])$ on $M$,
\item $\psi(a x_1, \cdots,x_{n-1})=a\psi(x_1, \cdots, x_{n-1})$, for all $a\in A$ and $x_1, \cdots,x_{n-1}\in L$,
\item $\psi(x_1, \cdots,x_{n-1})(a m)=a\psi(x_1, \cdots,x_{n-1})(m)+\rho(x_1, \cdots,x_{n-1})am$, for all $x_1,\cdots,x_{n-1}\in L$, $a\in A$ and $m\in M$.
\end{enumerate}
\end{defi}

\begin{ex}
If $M=A$, since $\rho$ is a representation of $(L, [\c, \cdots, \c])$
over $A$ and  the conditions (2) and (3) are satisfied automatically by definition of the map $\rho$, then $(A,\rho)$ is a representation of $L$ .
\end{ex}

The following proposition is straightforward.
\begin{pro}
Let $(L, A,[\cdot,\cdots,\cdot],\rho)$ be a $n$-Lie Rinehart algebra. Then $(M, \psi)$ is a representation over $(L, A,[\cdot,\cdots,\cdot],\rho)$  if and only if $(L\oplus M, A, [\cdot,\cdots,\cdot]_{L\oplus M},\rho_{L\oplus M})$ is a $n$-Lie Rinehart algebra where:
\begin{align*}
&a(x_1+m_1)=ax_1+am_1,\\
&[x_1+m_1,\cdots,x_n+m_n]_{L\oplus M}=[x_1, \cdots, x_n]+\displaystyle\sum_{i=1}^n(-1)^{n-i}\psi(x_1, \cdots,\widehat{x_i},\cdots,x_n)m_i,\\
& \rho_{L\oplus M} (x_1+m_1,\cdots, x_{n-1}+m_{n-1})=\rho(x_1, \cdots,x_{n-1}),
\end{align*}

for any $x_1, \cdots, x_n\in L$, $m_1,\cdots, m_n\in M$ and $a\in A.$\\
\end{pro}
\begin{rmk}
For any $n$-Lie Rinehart algebra $(L, A, [\cdot, \cdots, \cdot], \rho)$, the map $\ad$ is not, in general, a representation of $L$.
\end{rmk}
 Let $K=\ker \rho=\{x\in L, \rho(x, L,\cdots,L)=0\}$. It is called the kernel of  $\rho$.  We have the following lemma.
 \begin{lem}
  The set  $K$   is an ideal of $L$.
\end{lem}
\begin{proof}
Thanks to  Eqs. \eqref{rep1}, \eqref{rep2},  for all   $x\in K$, $y_1,\cdots,y_{n-1},z_1,\cdots,z_{n-1}\in  L$,
\begin{align*}
     \rho(z_1,\cdots,z_{n-2},[x,y_1\cdots,y_{n-1}])&=\sum_{i=1}^{n-1}(-1)^{n-i}\rho(x,y_1,\cdots,\hat{y_i},\cdots,y_{n-1})
     \circ\rho(z_1,\cdots,z_{n-2},y_i)\\
     &=0.
\end{align*}
Therefore, $[x, L,\cdots,L]\subseteq K$. By Eq.\eqref{eq:Rinhart2}, for all $a\in  A$,
\begin{align*}
     & \rho(ax,y_1,\cdots,y_{n-1})= a\rho(x,y_1,\cdots,y_{n-1})=0.
     \end{align*}
     Then $A K\subset K$. Also we have $\rho(K,L,\cdots,L)(A)L\subset K$, indeed
     \begin{align*}
      & \rho(x,y_1,\cdots,y_{n-2})ay_{n-1}=0\in K.
\end{align*}
Therefore, $K$ is an ideal of the  $n$-Lie Rinehart algebra $(L,A,[\cdot,\cdots,\cdot],\rho)$.
\end{proof}

 \begin{lem}
 Let  $(L,A,[\cdot,\cdots,\cdot],\rho)$ be a $n$-Lie Rinehart algebra. Then $(K, \ad)$ is a representation of $L$
 \end{lem}
 \begin{proof}
 It is clear that $\ad$  is a representation of the $n$-Lie algebra $(L,[\cdot,\cdots,\cdot])$ on $K$.
 For any $x_1,\cdots,x_{n-1}\in L$, $x_n\in K$ and  $a\in A$, we have
 \begin{align*}
  \ad(ax_1,\cdots,x_{n-1})x_n&=[ax_1,\cdots,x_n]\\
  &=a[x_1,\cdots,x_n]+(-1)^{n-1}\rho(x_2,\cdots,x_n)ax_1\\
  &=a[x_1,\cdots,x_n]\\
  &=a \big(\ad(x_1,\cdots,x_{n-1})x_n\big)
 \end{align*}
 and
 \begin{align*}
  \ad(x_1,\cdots,x_{n-1})(ax_n)&=[x_1,\cdots,ax_n]\\
  &=a[x_1,\cdots,x_n]+\rho(x_1,\cdots,x_{n-1})ax_n\\
  &=a \big(\ad(x_1,\cdots,x_{n-1})x_n\big)+\rho(x_1,\cdots,x_{n-1})ax_n.
 \end{align*}
 Then $(K, \ad)$ is a representation of $L$.
 \end{proof}
 Let $(L,A, [\cdot,\cdots,\cdot ], \rho)$ be an $n$-Lie Rinehart algebra and $(M,\psi)$ be a representation on $L$. Define  $M^\ast$  to be the set of $A$-linear maps $f:M\to A$,  with $A$-bilinear map $ \langle\cdot,\cdot\rangle:M^\ast\times M\to A$ given by
 $$\langle f,m\rangle=f(m),$$
 and define  $\psi^\ast:L\wedge\cdots\wedge L\to End(M^\ast)$   by
$$
\langle\psi^\ast(x_1,\cdots,x_{n-1})f,m\rangle=-\langle f,\psi(x_1,\cdots,x_{n-1})m\rangle
$$
for all $f\in M^\ast, m\in M, x_i\in L$.

\begin{thm}
 With the above notation,  the pair $(M^\ast, \psi^\ast)$ is a representation on $L$.
\end{thm}

 \begin{proof}
 For any $x_1,\cdots,x_{n-1},y_1,\cdots,y_{n-1}\in L$, $f\in M^\ast$, $m\in M$ and  $a\in A$   we have
 \begin{align*}
   &  \langle  \psi^\ast (x_1,\cdots,x_{n-1})\circ  \psi^\ast(y_1,\cdots,y_{n-1})f-  \psi^\ast(y_1,\cdots,y_{n-1})\circ   \psi^\ast(x_1,\cdots,x_{n-1})f,m\rangle\\
   =&\langle f,   \psi(y_1,\cdots,y_{n-1})\circ   \psi (x_1,\cdots,x_{n-1})m-  \psi (x_1,\cdots,x_{n-1})\circ  \psi (y_1,\cdots,y_{n-1})m\rangle\\
   =&-\langle f,   \psi (x_1,\cdots,x_{n-1})\circ  \psi (y_1,\cdots,y_{n-1})m-  \psi(y_1,\cdots,y_{n-1})\circ   \psi (x_1,\cdots,x_{n-1})m\rangle\\
   =&-\langle f,\displaystyle\sum_{i=1}^{n-1}  \psi(y_1,\cdots,y_{i-1},[x_1,\cdots,x_{n-1},y_i],y_{i+1},\cdots,y_{n-1})m\rangle\\
   =&\langle\displaystyle\sum_{i=1}^{n-1}  \psi^\ast(y_1,\cdots,y_{i-1},[x_1,\cdots,x_{n-1},y_i],y_{i+1},\cdots,y_{n-1})f,m\rangle.
 \end{align*}
 Then Eq. \eqref{rep1} is verified for $ \psi^\ast$. Similarly, we can check Eq.$\eqref{rep2}$.
 Therefore $(M^\ast,  \psi^\ast)$ is a representation  of the $n$-Lie algebra  $(L,  [\cdot,\cdots,\cdot])$.
Furthermore,
 \begin{align*}
     \langle\psi^*(ax_1,\cdots,x_{n-1})f,m\rangle=&-\langle f,\psi(ax_1,\cdots,x_{n-1})m\rangle\\
     =&-a\langle f, \psi(x_1,\cdots,x_{n-1})m\rangle\\
     =& a\langle \psi^*(x_1,\cdots,x_{n-1})f,m\rangle\\
     =& \langle a\psi^*(x_1,\cdots,x_{n-1})f,m\rangle,
 \end{align*}
 and

 \begin{align*}
     \langle  \psi^\ast(x_1,\cdots,x_{n-1})(af),m\rangle=&-\langle af,  \psi(x_1,\cdots,x_{n-1})m\rangle\\
     =&-\langle f,  a \psi(x_1,\cdots,x_{n-1})m\rangle\\
      =&-\langle f,   \psi(x_1,\cdots,x_{n-1})(am)\rangle+ \langle f,  \rho(x_1,\cdots,x_{n-1})am\rangle\\
      =&\langle \psi^\ast(x_1,\cdots,x_{n-1}) f,  am\rangle+ \langle \rho(x_1,\cdots,x_{n-1})af,  m\rangle\\
      =&\langle (a \psi^\ast(x_1,\cdots,x_{n-1})+\rho(x_1,\cdots,x_{n-1})a)f,  m\rangle.
 \end{align*}
 Then, $(M^\ast, \psi^\ast)$ is a representation on $L$.
 \end{proof}

The following corollary is straightforward.

\begin{cor}
 Let $(L,A, [\cdot,\cdots,\cdot ], \rho)$ be an $n$-Lie Rinehart algebra. Then $(K^\ast, \ad^\ast)$ is a representation of  $L$
where $\ad^\ast: \wedge^{n-1} L\to End(K^\ast)$ by
$$
\langle\ad^\ast(x_1,\cdots,x_{n-1})f,y\rangle=-\langle f,\ad(x_1,\cdots,x_{n-1})y\rangle.
$$
\end{cor}

  \

Let  $(M,\psi)$ be a representation of the $n$-Lie Rinehart algebra  $(L, A, [\cdot,\cdots,\cdot],\rho)$ and  $C^{p}(L, M)$ the space of all linear maps   $f: \wedge^{n-1}L\o\stackrel{(p-1)}{\ldots}\o\wedge^{n-1}L \wedge L\rightarrow M $ satisfying the conditions below:
\begin{enumerate}
\item $f(x_1, \dots,x_i,x_{i+1}, \dots, x_{(n-1)(p-1)+1})=-f(x_1,  \dots,x_{i+1},x_i, \dots, x_{(n-1)(p-1)+1})$,  for  all $x_i\in L,  1\leq i \leq (n-1)(p-1)+1$,
\item $f(x_1,\c \c \c, a x_i, \c \c \c, x_{(n-1)(p-1)+1})=af(x_1,\c \c \c, x_i, \c \c \c, x_{(n-1)(p-1)+1})$,  for  all $x_i\in L,  1\leq i \leq 2n+1$  and $a\in A$.
\end{enumerate}
Next we consider the \emph{cochains} as $\mathbb{Z}_{+}$-graded space of $\mathbb{K}$-modules
\begin{eqnarray*}
C^{\ast}(L, M):=\oplus_{p\geq 0}C^{p}(L, M).
\end{eqnarray*}

Define the \emph{coboubdaries} as $\mathbb{K}$-linear maps $\delta:  C^{p}(L, M)\rightarrow C^{p+1}(L, M)$ given by
\begin{eqnarray*}
&& \delta f(X_1,\c  \c  \c, X_{p},z)\\
&=& \sum_{1\leq i<k\leq p}(-1)^i f\Big(X_1,\cdots,\hat{X}_i,\cdots,X_{k-1},[X_i,X_k]_F,X_{k+1},\cdots,X_{p},z\Big)\\
&&+\sum_{i=1}^p(-1)^i f(X_1,\cdots,\hat{X}_i,\cdots,X_{p},
\ad(X_i)(z))\\
&&+\sum_{i=1}^p(-1)^{i+1}\psi(X_i) f(X_1,\cdots,\hat{X}_i,\cdots,X_{p},z)\\
&&+\sum_{i=1}^{n-1}(-1)^{n+p-i+1}\psi(x^1_p,x^2_p,\cdots,\hat{x^i_p},\cdots,x^{n-1}_p,z)
f(X_1,\cdots,X_{p-1},x^i_p) ,
\end{eqnarray*}
for all $X_i=(x^1_i,x^2_i,\cdots,x^{n-1}_i)\in \mathcal L$ and $z\in L.$

Since the proof of the following proposition is similar to the proofs in \cite{Ammar&Mabrouk&makhlouf,BenHassine&Chtioui&Mabrouk&Silvestrov,Mikolaj}, we decided not include it.
\begin{proof}
The proof of this result can be deduced from \cite{Ammar&Mabrouk&makhlouf,BenHassine&Chtioui&Mabrouk&Silvestrov,Mikolaj}
 \end{proof}

A $p$-cochain $f\in C^{p}(L,M)$ is called a { $p$-cocycle} if $\delta f=0$. A $p$-cochain $f\in C^{p}(L,M)$ is called a { $p$-coboundary} if $f=\delta g$ for some $g\in C^{p-1}(L,M)$. Denote by $\mathcal{Z}^p(L,M)$ and $\mathcal{B}^p(L,M)$ the sets of $p$-cocycles and  $p$-coboundaries respectively. We define the $p$-th cohomology group
$\mathcal{H}^p(L,M)$ to be $\mathcal{Z}^p(L,M)/\mathcal{B}^p(L,M)$.

\section{Extensions of $n$-Lie  Rinehart algebras}\label{Extensions}
In this section, we deal with extensions of $n$-Lie  Rinehart algebras. \\
An algebra $\tilde{L}$ is  an extension of $n$-Lie  Rinehart algebras $L$  by $V$ if there is an exact sequence 
$$0 \longrightarrow V  \xrightarrow{\ \ i\ \ } \tilde{L}  \xrightarrow{\ \ \pi\ \ } L \longrightarrow 0$$

\subsection{Central extensions of $n$-Lie  Rinehart algebras}
\begin{defi}
A central extension of an $n$-Lie Rinehart algebra $(L,A,[\cdot,\cdots,\cdot],\rho)$  is an extension in which the center of $\tilde{L}$ contains $V$.
\end{defi}

\begin{ex}
 Let $(L,A,[\cdot,\cdots,\cdot],\rho)$ be an $n$-Lie  Rinehart algebra, $V$ be a one dimensional $A$-vector space generated by $e$  and $\varphi:L^{\otimes n} \to A$ is a linear map.
 Define the central extension   $\tilde L=L\oplus V$ by the bracket
\begin{equation}\label{7.1}
[\tilde{x}_1,\cdots,\tilde{x}_n]_{\widetilde{L}}=[x_1,\cdots,x_n]+\varphi(x_1,\cdots,x_n)e
\end{equation}
and the representation  $\tilde \rho :\wedge^{n-1}\tilde L\to Der(A)$  is defined by
\begin{equation}\label{7.2}
 \tilde\rho(\tilde{x}_1,\cdots,\tilde{x}_{n-1})=\rho({x}_1,\cdots,{x}_{n-1}),
\end{equation}
where $A$ acts on $\tilde L$ via
$$
a(x_i+\lambda_i e)= ax_i+\lambda_i ae
$$
for all $\tilde x_i=x_i+\lambda_i e\in \tilde L$, $\lambda_i\in \mathbb K$ and $a\in A$.
Then $(\tilde L, A, [\cdot,\cdots,\cdot]_{\tilde L},\tilde \rho)$ is an $n$-Lie Rinehart algebra if and only if $\varphi$ is a $2$-cocycle.

Indeed, it is clear that $[\cdot,\cdots,\cdot]_{\tilde L}$  is $n$-linear and skew-symmetric if and only if $\varphi$ is  $n$-linear and skew-symmetric map.  Furthermore, the condition Eq.\eqref{eq:Rinhart1} holds if and only if $\varphi$ is $A$-linear map.
 In fact,
\begin{align*}
    &[\tilde x_1, \cdots, a\tilde x_n]_{\tilde L}-a[\tilde x_1, \cdots, \tilde x_n]-\tilde\rho(\tilde x_1,\cdots,\tilde x_{n-1})a \tilde x_n\\
    =&[x_1,\cdots,ax_n]+\varphi(x_1,\cdots,ax_n)e-a[x_1,\cdots,x_n]-a\varphi(x_1,\cdots,x_n)e-\rho(x_1,\cdots,x_{n-1})a x_n\\
    =&\varphi(x_1,\cdots,ax_n)e-a\varphi(x_1,\cdots,x_n)e.
\end{align*}Then, $\varphi$ is  a $2$-cochain.
  The fondamental identity  \eqref{Fundamental.identity} can be written,
for  $\tilde{x}_i=x_i+\lambda_i e\in \tilde{L}\ $, $\tilde{y}_i=y_i+\nu_i e\in \tilde{L},\ \ 1\leq i\leq n$, as
  $$
  \big[\tilde{x}_1,\cdots,\tilde{x}_{n-1},[\tilde{y}_1,\cdots,\tilde{y}_{n}]_{\tilde{L}}\big]_{\tilde{L}}~-  \sum_{i=1}^{n}\big[\tilde{y}_1,\cdots,\tilde{y}_{i-1},[\tilde{x}_1,\cdots,\tilde{x}_{n-1},y_i]_{\tilde{L}}.
 \tilde{y}_{i+1},\cdots,\tilde{y}_n\big]_{\tilde{L}}=0.
  $$
Using Eq. \eqref{7.1} and  the fact that $e$ is a central element, one gets
   \begin{align}\label{7.3}& \varphi\big(x_1,\cdots,x_{n-1},[y_1,\cdots,y_{n}]\big)-  \sum_{i=1}^{n}\varphi\big(y_1,\cdots,y_{i-1},[x_1,\cdots,x_{n-1},y_i],y_{i+1},\cdots,y_n\big)=0.
  \end{align}
  \item  The previous equation, may be written as $$\delta^2\varphi(X,Y,y_n)=0$$
  where  $X=(x_1, \cdots, x_{n-1}),\ Y=(y_1,\cdots, y_{ n-1})\in \mathcal L$.

\end{ex}
\subsection{$T_{\theta}$-extension of $n$-Lie  Rinehart algebras}

Let $(L, A,[\cdot, \cdots, \cdot], \rho)$ be an $n$-Lie Rinehart algebra and $(M ,\psi)$ be a representation over $L$.
If $\theta:\wedge^{n} L\rightarrow M$ is a $2$-cochain. Then $\theta$ is a $2$-cocycle associated with the module $(M,\psi)$ if it satisfies, for every $x_1,\cdots, x_n, y_1\cdots,y_{n-1},$
 \begin{eqnarray}\label{2-cocycle}
  &&\theta(y_{1},\cdots,y_{n-1},[x_1,\cdots,x_n])=\displaystyle\sum_{i=1}^n\theta(x_1,\cdots,[y_{1},\cdots,y_{n-1},x_i],\cdots,x_{n})\\
&&  +\displaystyle\sum_{i=1}^n(-1)^{n-i}\psi(x_1\cdots,\widehat{x_i},\cdots,x_n)\theta(y_{1},\cdots,y_{n-1},x_i)-\psi(y_{1},\cdots,y_{n-1})\theta(x_1,\cdots,x_n).\nonumber
  \end{eqnarray}

We recall the following lemma from  \cite{Bai&Li2012}.
\begin{lem}\label{lem1}
Let $(L, [\cdot, \cdots, \cdot])$ be a $n$-Lie  algebra and $(M ,\psi)$ be a representation of   $L$. Let $\theta:\wedge^{n} L\rightarrow M$ and   $[\cdot,\cdots,\cdot]_\theta:\wedge^n(L\oplus M) \rightarrow L \oplus M$ given by
{\begin{align}\label{crochet-theta}
 [x_1+m_1,\cdots,x_n+m_n]_\theta=&[x_{1},\cdots,x_{n}]+\theta(x_{1},\cdots,x_{n})\\
 &+\displaystyle\sum_{i=1}^n(-1)^{n-i}\psi(x_1\cdots,\widehat{x_i},\cdots,x_n)m_i,\nonumber
 \end{align}}
 for all $x_1,\cdots,x_n\in L$ and $m_1,\cdots,m_n\in M$.
 Then $(L\oplus M,[\cdot,\cdots,\cdot]_{\theta})$ is $n$-Lie  algebra if and only if  $\theta$ is a $2$-cocycle  associated with the module $(M,\psi)$.
 \end{lem}

 Now, we are able to construct new $n$-Lie Rinehart algebras by adding the module $(M,\psi)$ as follows.
\begin{thm}Let $(L, A,[\cdot, \cdots, \cdot], \rho)$ be a $n$-Lie Rinehart algebra and $(M ,\psi)$ be a module over $L$.  Then $\theta:\wedge^n (L\oplus M)\to L \oplus M$ is a $2$-cocycle if and only if $(L\oplus M, A, [\cdot,\cdots,\cdot]_{\theta},\rho_\theta)$
 is a $n$-Lie Rinehart algebra, where  $[\cdot,\cdots,\cdot]_\theta$ is defined in Eq. \eqref{crochet-theta}
 and $\rho_\theta:\wedge^{n-1}(L\oplus M)\to Der(A)$  given by
 \begin{align}
 \rho_\theta(x_1+m_1,\cdots,x_{n-1}+m_{n-1})=\rho(x_1,\dots,x_{n-1}),
 \end{align}
 and $A$  acts on $L\oplus M$ via
 \begin{equation*}
a(x_1+m_1)=ax_1+am_1,
\end{equation*}
 for all $x_1,\cdots,x_n\in L$, $m_1,\cdots,m_n\in M$ and $a\in A$.
\end{thm}

\begin{proof}
Suppose that $\theta$ is a $2$-cocycle. Then for all $x_1,\cdots,x_n\in L$, $m_1,\cdots,m_n\in M$ and $a,b\in A$, it is obvious to show that
 $$\rho_\theta(x_1,\cdots,x_{n-1})\in Der(A)$$ and
 $$\rho_\theta(x_1,\cdots,ax_i,\dots,x_{n-1})=a\rho_\theta(x_1,\cdots,x_i,\dots,x_{n-1})
 $$
 Now, we will show the compatibility condition Eq.\eqref{eq:Rinhart1}.
 \begin{align*}
  &[x_1+m_1,\cdots,a(x_n+m_n)]_\theta=[x_1+m_1,\cdots,ax_n+am_n]_\theta\\
  =&[x_{1},\cdots,ax_{n}]+\theta(x_{1},\cdots,ax_{n})
 +\displaystyle\sum_{i=1}^{n-1}(-1)^{n-i}\psi(x_1\cdots,\widehat{x_i},\cdots,ax_n)m_i\\
 &+\psi(x_1\cdots,\widehat{x_i},\cdots,x_{n-1})(am_n)\\
 =&a[x_{1},\cdots,x_{n}]+\rho(x_1,\cdots,x_{n-1})ax_n+\theta(x_{1},\cdots,ax_{n})+a\displaystyle\sum_{i=1}^{n-1}(-1)^{n-i}\psi(x_1\cdots,\widehat{x_i},\cdots,x_n)m_i\\
 &+a\psi(x_1,\cdots,x_{n-1})m_n+\rho(x_1,\cdots,x_{n-1})am_n\\
 =&a\big([x_{1},\cdots,x_{n}]+\theta(x_{1},\cdots,x_{n})+\displaystyle\sum_{i=1}^{n}(-1)^{n-i}\psi(x_1\cdots,\widehat{x_i},\cdots,x_n)m_i\big)\\
 &+ \rho(x_1,\cdots,x_{n-1})a(x_n+m_n)\\
 =&a[x_1+m_1,\cdots,x_n+m_n]_\theta+\rho_\theta(x_1+m_1,\cdots,x_{n-1}+m_{n-1})a(x_n+m_n).
 \end{align*}
 Conversely, if  $(L\oplus M, A, [\cdot,\cdots,\cdot]_{\theta},\rho_\theta)$  is a $n$-Lie Rinehart algebra then  $(L\oplus M, [\cdot,\cdots,\cdot]_{\theta})$ is an $n$-Lie algebra. Using lemma \ref{lem1}, we have  $\theta$ is a $2$-cocycle  associated with the module $(M,\psi)$.
 \end{proof}

\begin{defi}
Denote the  $n$-Lie Rinehart algebra $(L\oplus M, A, [\cdot,\cdots,\cdot]_{\theta},\rho_\theta)$   by $T_{\theta}(L)$. It is called the $T_{\theta}$-extension of $(L, A, [\cdot, \cdots, \cdot], \rho)$ by the $L$-module $M$.
\end{defi}

\begin{lem}
Let $(L, A, [\cdot, \cdots, \cdot]_{L}, \rho)$ be a $n$-Lie  Rinehart algebra
and $(M ,\psi)$ be an $L$-module. For every $1$-cochain
$f:L\rightarrow M$, the  skew-symmetric $n$-map
$\theta_{f}:\wedge^n L\rightarrow M$ given by
\begin{eqnarray}
\theta_{f}(x_1,\cdots,x_n)=f([x_1,\cdots,x_n])-\displaystyle\sum_{i=1}^n(-1)^{n-i}\psi(x_1,\cdots,\widehat{x_i},\cdots,x_n)f(x_i),\end{eqnarray}
for all $x_1,\cdots,x_n\in L$, is a $2$-cocycle associated with $\psi$.
\end{lem}

\begin{proof}
It is clear that $\theta_f=\delta f$, then $\theta_f$ is a $2$-cocycle.
\end{proof}
With the above notation we have the following theorem.
\begin{thm}
Let $(L, A, [\cdot, \cdots, \cdot], \rho)$ be an $n$-Lie Rinehart algebra
and $(M ,\psi)$ be a representation of  $L$. Define the map $$\Phi:T_{\theta}(L)\rightarrow
T_{\theta+\theta_{f}}(L) ~~\text{given~ by}~\Phi(x+m)=x+f(x)+m,\quad \forall x\in L,m\in M.
$$
Then the pair $(\Phi,Id_A)$ is an $n$-Lie Rinehart algebra  isomorphism.
\end{thm}
\begin{proof}
It is clear that $\Phi$ is a bijection from $L\oplus M$ to $L\oplus M$. For all  $x_{1},\cdots,x_{n}\in L$, $m_{1},\cdots,m_{n}\in M$,
\begin{eqnarray*}
&&\Phi([x_{1}+m_{1}\cdots,x_{n}+m_{n}]_{\theta})\\
&=&\Phi([x_{1},\cdots,x_{n}]+\theta(x_{1},\cdots,x_{n})+\displaystyle\sum_{i=1}^n(-1)^{n-i}\psi(x_1\cdots,\widehat{x_i},\cdots,x_n)m_i)\\
&=&[x_{1},\cdots,x_{n}]+\theta(x_{1},\cdots,x_{n})+f([x_{1},\cdots,x_{n}])+\displaystyle\sum_{i=1}^n(-1)^{n-i}\psi(x_1\cdots,\widehat{x_i},\cdots,x_n)m_i\\
&=&[x_{1},\cdots,x_{n}]+(\theta+\theta_{f})(x_{1},\cdots,x_{n})+\displaystyle\sum_{i=1}^n(-1)^{n-i}\psi(x_1\cdots,\widehat{x_i},\cdots,x_n)m_i\\
&&+\displaystyle\sum_{i=1}^n(-1)^{n-i}\psi(x_1\cdots,\widehat{x_i},\cdots,x_n)f(x_i)\\
&=&[x_{1}+f(x_{1})+m_{1},\cdots,x_{n}+f(x_{n})+m_{n}]_{\theta+\theta_{f}}\\
&=&[\Phi(x_{1}+m_{1}),\cdots,\Phi(x_{n}+m_{n})].
 \end{eqnarray*}
It is straightforward to check the identity $(1)$ of the Definition \ref{homorphism-Rinehart}.
 For all $x_1,\cdots,x_{n-1}\in L$, $m_1,\cdots,m_{n-1}\in M$ and $a\in A$ we have

 \begin{align*}
     \rho_{\theta}(x_1+m_1,\dots,x_{n-1}+m_{n-1})(a)& = \rho(x_1,\dots,x_{n-1})(a)\\
     &=\rho_{\theta+\theta_f}(x_1+f(x_1)+m_1,\cdots,x_{n-1}+f(x_{n-1})+m_{n-1})(a)\\
     &=\rho_{\theta+\theta_f}(\Phi(x_1+m_1),\cdots,\Phi(x_{n-1}+m_{n-1}))(a).
 \end{align*}

\end{proof}
\section{Crossed modules of $n$-Lie Rinehart algebras}\label{Crossed}
In  this section,  we define a notion of crossed modules for  $n$-Lie Rinehart algebras   which is related to the third dimensional cohomology group.

\begin{defi}\label{L-act-M} Let $(L, A, [\cdot,\cdots,\cdot],\rho)$ be an $n$-Lie Rinehart algebra, $(M,[\cdot,\cdots,\cdot]_M, \psi)$ be an  $n$-Lie $A$-algebra and  $\psi:  \wedge^{n-1} L\rightarrow End(M)$ be a linear map. We say that $M$ acts on $L$ if  $(M,\psi)$ is a representation of $L$ and the following condition holds
$$
\psi(x_1,\cdots,x_{n-1})[m_1,\cdots,m_n]_M= \displaystyle\sum_{i=1}^n[m_1,\cdots,\psi(x_1,\cdots,x_{n-1})m_i,\cdots,m_n]_M
$$
for all $x_1,\cdots,x_{n-1}\in L, m_1,\cdots,m_n\in M$.
The tuple $(M,[\cdot,\cdots,\cdot],\psi)$ will be called also $n$-Lie $A$-algebra $L$-module.
\end{defi}

Now for a $n$-Lie Rinehart algebra $(L, A, [\cdot,\cdots,\cdot],\rho)$ and $n$-Lie $A$-algebra $L$-module $(M, [\cdot,\cdots, \cdot]_M)$. Consider the direct sum $L\oplus M$ as an $A$-module:
$a(x+ m) = ax+am$. Define the anchor $\rho_{L\oplus M} : \wedge^{n-1}(L\oplus M) \to Der(A)$ by
$$
\rho_{ L\oplus M}(x_1+ m_1,\cdots,x_{n-1}+ m_{n-1}) = \rho(x_1,\cdots,x_{n-1})$$
and the bracket by
$$\small
[x_1+ m_1, \dots,x_n+ m_n]_{L\oplus M} = [x_1, \cdots,x_n] + [m_1, \cdots,m_n]_M+ \displaystyle\sum_{i=1}^n(-1)^{n-i}\psi(x_1,\cdots,x_{n-1})m_i
$$
for all $x_i\in L$ and $m_i\in M$, $i=1,\cdots,n.$

We have the following theorem.
\begin{thm}
Using the above notation, the tuple $(L\oplus M,A, [\cdot,\cdots,\cdot]_{ L\oplus M},\rho_{ L\oplus M})$  is a $n$-Lie Rinehart algebra if and only if the conditions of Definition \ref{L-act-M} hold. It is called a semidirect product of $n$-Lie Rinehart algebras, denoted by $L \rtimes M.$
\end{thm}

The following definition gives a notion of crossed modules for $n$-Lie Rinehart algebras.
\begin{defi}\label{definition-crossed}
Let $(L, A,[\cdot,\cdots,\cdot],\rho)$ be a $n$-Lie Rinehart algebra, $(M, [\cdot,\cdots,\cdot], \psi)$ be a $n$-Lie $A$-algebra $L$-module. If a $n$-Lie
algebra homomorphism $\partial : M \to L$ satisfies:
\begin{enumerate}
    \item $\partial(\psi(x_1,\cdots,x_{n-1})m) = [x_1,\cdots,x_{n-1}, \partial m], \forall m \in M, x_1,\cdots,x_{n-1}\in L,$\label{def11}
\item $\psi(\partial m_1,\cdots, \partial m_{n-1})m = [m_1, \cdots,m_{n-1}, m], \forall m_1,\cdots, m_{n-1}, m \in  M,$\label{def12}
\item $\partial(am) = a\partial(m), \forall a \in  A, m \in M,$
\item $\rho(\partial(m_1),\cdots, \partial(m_{n-1}))(a) = 0, \forall a \in A, m_1,\cdots, m_{n-1} \in  M.$
\end{enumerate}
Then $(M, A, \psi, \partial)$ is called a crossed module of $n$-Lie Rinehart algebra $(L, A,[\cdot,\cdots,\cdot],\rho)$.
\end{defi}

\begin{rmk}\begin{itemize}\item[]
    \item[a)] Property 1. in the above definition means that the morphism  $\partial$ is equivariant with
respect to the $L$-action via $\psi$ on $M$ and the adjoint action on $L$.
\item[b)]  Property 2. is called Peiffer identity.
\end{itemize}

\end{rmk}
\begin{rmk}
\begin{enumerate}\item[]
\item  To each crossed module of  $n$-Lie Rinehart algebra $(L, A,[\cdot,\cdots,\cdot],\rho)$ and $\partial: M \to L$, we can  associates
a  exact sequence defined by
$$ \quad \quad 0 \longrightarrow {\mathfrak n} \xrightarrow{\ \ i\ \ }  M  \xrightarrow{\ \ \partial \ \ } L\xrightarrow{\ \ \pi\ \ }  \mathfrak p \longrightarrow 0.$$
where $\ker(\partial)= {\mathfrak n}$ and $\mathfrak p=coker(\partial)=L/Im(\partial).$
 \item
Let $(M, A, \psi, \partial)$ is called a crossed module of $n$-Lie Rinehart algebra $(L, A,[\cdot,\cdots,\cdot],\rho)$. Then we have
    \subitem{\bf a)} $Im(\partial)$ is an ideal of $L$.\label{lem11}
    \subitem{\bf b)}  $\ker(\partial)$ is a central ideal of $M$.\label{lem12}
\subitem{\bf c)}  Lifting elements of $\mathfrak p$ to $L$, the action of $L$ on $M$ induces an outer action of
$\mathfrak p$ on $M$   via $\overline\psi(\pi(x_1), \cdots,\pi(x_{n-1}))m = \psi(x_1,\cdots, x_{n-1})m,\forall x_1,\cdots,x_{n-1}\in L, m\in  M$.
\end{enumerate}
\end{rmk}

\begin{ex}
 Let $(L, A, [\cdot,\cdots,\cdot], \rho)$ be a $n$-Lie Rinehart algebra and $I \varsubsetneq L$ be an ideal  of  $(L,A,[\cdot,\cdots,\cdot],\rho)$. Then $(I, A, \mbox{ad},i)$ is a crossed module of the $n$-Lie Rinehart algebra $(L, A, [\cdot,\cdots,\cdot], \rho)$, where $i: I\rightarrow L$, $i(x)=x, \forall x\in I.$
\end{ex}

\begin{ex} \label{thm:inclusion} Let $(L, A, [\cdot,\cdots,\cdot]_L, \rho)$ and $(L', A, [\cdot,\cdots,\cdot]_{L'},\rho')$ be an $n$-Lie Rinehart algebras, and $f: L\rightarrow L'$ be an $n$-Lie Rinehart algebras homomorphism,
$\partial: \ker(f)\to L$ be the including mapping, that is, $\partial(u)=u$, for all $u\in \ker(f)$. Then $(\ker(f), A, \mbox{ad}, \partial)$ is a crossed module of $n$-Lie Rinehart algebra $(L, A, [\cdot,\cdots,\cdot]_L,\rho)$.
\end{ex}

 \begin{ex}
  Let $(L, A, [\cdot,\cdots,\cdot], \rho)$ be a $n$-Lie Rinehart algebra,  and $(M,\psi)$ be an $L$-module. Then $(M, A, \psi,0)$ is a crossed module of  $(L, A, [\cdot,\cdots,\cdot], \rho)$.
 \end{ex}

\begin{thm}
Let $(L, A,[\cdot,\cdots,\cdot],\rho)$ be a $n$-Lie Rinehart algebra and $(M, [\cdot,\cdots,\cdot], \psi)$ be a $n$-Lie $A$-algebra $L$-module. Then a morphism $\partial : M\to L$ is a crossed module of $n$-Lie Rinehart algebra $L$ if and only if
the maps
$$(Id_L+\partial):L\rtimes M\to L\rtimes L~~~\text{and}~~(\partial+Id_M):M\rtimes M\to L\rtimes M $$
are homomorphisms of $n$-Lie Rinehart algebras.
\end{thm}
\begin{proof}
For all $x_1,\cdots,x_n\in L$ and $m_1,\cdots,m_n\in M$, we have
\begin{align*}
    &(Id_L+\partial)[x_1+m_1,\cdots,x_n+m_m]_{L\rtimes M}\\
    =&(Id_L+\partial)([x_1, \cdots,x_n] + [m_1, \cdots,m_n]_M+ \displaystyle\sum_{i=1}^n(-1)^{n-i}\psi(x_1,\cdots,x_{n-1})m_i)\\
    =&[x_1, \cdots,x_n]+\partial[m_1, \cdots,m_n]_M+\displaystyle\sum_{i=1}^n(-1)^{n-i}\partial(\psi(x_1,\cdots,x_{n-1})m_i)).
\end{align*}
On the others hand,
\begin{align*}
   & [(Id_L+\partial)(x_1+m_1),\cdots,(Id_L+\partial)(x_n+m_m)]_{L\rtimes M}\\
   =& [x_1+\partial m_1,\cdots,x_n+\partial m_n]_{L\rtimes M}\\
    =&[x_1, \cdots,x_n] + [\partial m_1, \cdots,\partial m_n]_M+ \displaystyle\sum_{i=1}^n(-1)^{n-i}\psi(x_1,\cdots,x_{n-1})\partial m_i.
\end{align*}
Then, $(Id_L+\partial)$ is homomorphism of $n$-Lie algebra if and only if the condition \ref{def11}.  in Definition \ref{definition-crossed} is satisfied. Similarly, with the same computation we can check the other conditions.
\end{proof}

\begin{defi}

Two crossed modules $\partial:M\to L$ (with action $\psi$) and $\partial':M'\to L'$  (with action $\psi'$) such that
$\ker (\partial) = \ker (\partial') =: \mathfrak n$ and $coker (\partial)=coker (\partial')=: \mathfrak p$ are called elementary equivalent if there is a morphism of $n$-Lie $A$-algebras $\delta: M\to M'$ and a morphism of $n$-Lie Rinehart algebras $\gamma:L\to L'$ which are compatible with the actions,
meaning
$$
\delta(\psi(x_1,\cdots,x_{n-1})(m)) = \psi'(\gamma(x_1),\cdots,\gamma(x_{n-1}))(\delta(m)),
$$

for all $x_1,\cdots, x_{n-1}\in L$ and all $m \in M$, and such that the following diagram is commutative:
$$
\xymatrix{
(\mathcal{E}):0 \ar[r] & \mathfrak n \ar[d]^{id_\mathfrak n} \ar[r]^{i} & {M} \ar[d]^{\delta} \ar[r]^{\partial} & {L} \ar[d]^{\gamma} \ar[r]^{\pi} &  {\mathfrak p} \ar[d]^{id_{\mathfrak p}} \ar[r] & 0 \\
(\mathcal{E'}):0 \ar[r] & \mathfrak n  \ar[r]^{i'} &M' \ar[r]^{\partial'} & {L'}  \ar[r]^{\pi'} &  {\mathfrak p}  \ar[r] & 0}
\vspace{.5cm}
$$
\end{defi}
Let us denote by $Cr_{mod}(\mathfrak n,\mathfrak p)$ the set of equivalence classes of $n$-Lie Rinehart algebra
crossed modules with respect to fixed kernel $\mathfrak n$ and fixed cokernel $\mathfrak p$.

\

In the following result, we will interest to ternary case. The situation for $n\geq 3$ is similar and we leave to the reader the routine modifications to establish it.
\begin{thm}
For any $3$-Lie Rinehart algebra $(L, A,[\cdot,\cdot,\cdot],\rho)$  and $(M, [\cdot,\cdot,\cdot]_M,\psi)$ be a $A$-algebra module on $L$,  there is a canonical map $$\Phi : Cr_{mod}(\mathfrak n,\mathfrak p)\to H^3(\mathfrak n,\mathfrak p).$$
\end{thm}
\begin{proof}
Let $\partial : M \to L$ be a crossed module of $3$-Lie Rinehart algebra.
Given a crossed extension
 $$
 (\mathcal{E}):\quad \quad 0 \longrightarrow {\mathfrak n}=\ker(\partial)\xrightarrow{\ \ i\ \ }  M  \xrightarrow{\ \ \partial \ \ } L\xrightarrow{\ \ \pi\ \ }  \mathfrak p=coker(\partial) \longrightarrow 0.$$
 The first step is to take a linear section $s$ of $\pi$ and to compute the failure of
$s$ to be a $3$-Lie Rinehart algebra homomorphism, i.e.
$$\alpha(x_1,x_2,x_3)=[s(x_1),s(x_2),s(x_3)]-s([x_1,x_2,x_3]), \forall x_1,x_2,x_3\in \mathfrak p.$$
It is clear that $\alpha$ is $3$-$A$-linear map and skew-symmetric. Indeed,
\begin{align*}
    \alpha(x_1,x_2,ax_2)=&[s(x_1),s(x_2),s(ax_3)]-s([x_1,x_2,ax_3])\\
    =&a[s(x_1),s(x_2),s(x_3)]+\psi(s(x_1),s(x_2))as(x_3)\\
    &-as([x_1,x_2,x_3])-s(\psi(x_1,x_2)ax_3)\\
    =&a([s(x_1),s(x_2),s(x_3)]-s([x_1,x_2,x_3])\\
    =&a\alpha(x_1,x_2,x_3).
\end{align*}
Since $\pi$ is a $3$-Lie Rinehart algebra homorphism and $\pi s=Id_L$, then it is obvious that $\pi(\alpha(x,y,z))=0$. So $\alpha(x_1, x_2,x_3) \in \ker(\pi) = Im(\partial)\subset L$.

  Choose a linear section $\sigma:Im(\partial)\to M,$ $\partial \sigma=Id_M$ and take
  \begin{align*}
      \beta(x_1,x_2,x_3)=\sigma \alpha(x_1,x_2,x_4).
  \end{align*}
  Now, define
  \begin{align*}
      h_\mathcal E(x_1,x_2, x_3, x_4, x_5)=&\beta([x_1,x_2, x_3], x_4, x_5)+\beta(x_3,[x_1,x_2, x_4], x_5)+\beta(x_1, x_2,[ x_3,x_4, x_5])\\
      +&\beta(x_3,x_4, [x_1, x_2,x_5]) +\tilde\psi(sx_1, sx_2)\beta(x_3,x_4, x_5)+\tilde\psi(sx_3,sx_4) \beta(x_1, x_2,x_5)\\
      -&\tilde\psi(sx_3,sx_5) \beta(x_1, x_2,x_4)+\tilde\psi(sx_4,sx_5) \beta(x_1, x_2,x_3)
  \end{align*}
  Since $\beta$ is $3$-$A$-linear then $ h_\mathcal E$ also it. Moreover,
  \begin{align*}
      \partial h_\mathcal E(x_1,x_2,x_3,x_4,x_5)=&\alpha([x_1,x_2, x_3], x_4, x_5)+\alpha(x_3,[x_1,x_2, x_4], x_5)+\alpha(x_1, x_2,[ x_3,x_4, x_5])\\
      &+\alpha(x_3,x_4, [x_1, x_2,x_5]) +\partial\tilde\psi(sx_1, sx_2)\beta(x_3,x_4, x_5)+\partial\tilde\psi(sx_3,sx_4) \beta(x_1, x_2,x_5)\\
      &-\partial\tilde\psi(sx_3,sx_5) \beta(x_1, x_2,x_4)+\partial\tilde\psi(sx_4,sx_5) \beta(x_1, x_2,x_3)\\
      =&[[sx_1,sx_2, sx_3], sx_4, sx_5]-s([[x_1,x_2, x_3], x_4, x_5])\\
      &+[sx_3,[sx_1,sx_2, sx_4], sx_5]-s([x_3,[x_1,x_2, x_4], x_5])\\
      &+[sx_1, sx_2,[ sx_3,sx_4, sx_5]]-s([x_1, x_2,[ x_3,x_4, x_5]])\\
      &+[sx_3,sx_4, [sx_1, sx_2,sx_5]]-s([x_3,x_4, [x_1, x_2,x_5]])\\
      &+[sx_1, sx_2,\alpha(x_3,x_4, x_5)]+[sx_3,sx_4, \alpha(x_1, x_2,x_5)]\\
      &-[sx_3,sx_5, \alpha(x_1, x_2,x_4)]+[sx_4,sx_5,\alpha(x_1, x_2,x_3)]\\
      =&[[sx_1,sx_2, sx_3], sx_4, sx_5]-s([[x_1,x_2, x_3], x_4, x_5])\\
      &+[sx_3,[sx_1,sx_2, sx_4], sx_5]-s([x_3,[x_1,x_2, x_4], x_5])\\
      &+[sx_1, sx_2,[ sx_3,sx_4, sx_5]]-s([x_1, x_2,[ x_3,x_4, x_5]])\\
      &+[sx_3,sx_4, [sx_1, sx_2,sx_5]]-s([x_3,x_4, [x_1, x_2,x_5]])\\
      &+[sx_1, sx_2,[sx_3,sx_4, sx_5]-s[x_3,x_4, x_5]]\\
      &+[sx_3,sx_4, [sx_1, sx_2,sx_5]-s[x_1, x_2,x_5]]\\
      &-[sx_3,sx_5, [sx_1, sx_2,sx_4]-s[x_1, x_2,x_4]]\\
     & +[sx_4,sx_5, [sx_1, sx_2,sx_3]-s[x_1, x_2,x_3]]\\
      &=0.
  \end{align*}
  Then, $Im(h_\mathcal E)\subset\ker\partial=\mathfrak n$. Thus, we define a
linear map $h_\mathcal E :\wedge^5 \mathfrak p\to \mathfrak n$ which is $3$-cochain. Routine calculations show that $\delta^3 h_\mathcal E=0.$ Then, $h_\mathcal E\in Z^3(\mathfrak p;\mathfrak n)$.
Define $$\Phi: Cr_{mod}(\mathfrak p;\mathfrak n)\to H^3(\mathfrak p;\mathfrak n),~~~~\Phi(\mathcal E)= |h_\mathcal E|.$$
Now, we will check  $\Phi$ is a well defined that is  the class of $h_\mathcal E$  does not depend on the sections $s, \sigma$ and if there is a map of crossed extensions $\mathcal E\to \mathcal E'$, then $|h_\mathcal E|=|h_{\mathcal E'}|$ in $H^3(\mathfrak p;\mathfrak n)$.\\
Let $\bar s:  \mathfrak p \to L$ be another linear section of $\pi$ and $\overline{h}_{\mathcal E}$ be the $3$-cocycle defined using
$\bar s$ instead of $s$. Since $s$ and $\bar s$ are both sections of $\pi$ there exists a linear map $g:\mathfrak p\to M$
with $\bar s- s = \partial g$. Then,
\begin{align*}
 &( h_\mathcal E-\overline{h}_{\mathcal E})(x_1,x_2, x_3, x_4, x_5)\\
 =&(\beta-\overline \beta)([x_1,x_2, x_3], x_4, x_5)+(\beta-\overline \beta)(x_3,[x_1,x_2, x_4], x_5)\\
 &+(\beta-\overline \beta)(x_1, x_2,[ x_3,x_4, x_5])-(\beta-\overline \beta)(x_3,x_4, [x_1, x_2,x_5]) \\
 &+ \psi(sx_1, sx_2)\beta(x_3,x_4, x_5)- \psi(\bar sx_1, \bar sx_2)\overline \beta(x_3,x_4, x_5)\\
 &+ \psi(sx_3,sx_4) \beta(x_1, x_2,x_5)- \psi(\bar sx_3,\bar sx_4) \overline \beta(x_1, x_2,x_5)\\
    &- \psi(sx_3,sx_5) \beta(x_1, x_2,x_4)+ \psi(\bar sx_3,\bar sx_5) \overline\beta(x_1, x_2,x_4)\\
    &+ \psi(sx_4,sx_5) \beta(x_1, x_2,x_3)  - \psi(\bar sx_4,\bar sx_5) \overline \beta(x_1, x_2,x_3) \\
    =&(\beta-\overline \beta)([x_1,x_2, x_3], x_4, x_5)+(\beta-\overline \beta)(x_3,[x_1,x_2, x_4], x_5)\\
 &+(\beta-\overline \beta)(x_1, x_2,[ x_3,x_4, x_5])-(\beta-\overline \beta)(x_3,x_4, [x_1, x_2,x_5]) \\
 &+ \psi(sx_1, sx_2)(\beta-\overline \beta)(x_3,x_4, x_5)- \psi(( s-\bar s)x_1, ( s- \bar s)x_2)\overline \beta(x_3,x_4, x_5)\\
 &+ \psi((s-\bar s)x_1, sx_2)\overline \beta(x_3,x_4, x_5)+ \psi(sx_1, (s-\bar s)x_2)\overline \beta(x_3,x_4, x_5)\\
 &+ \psi(sx_3, sx_4)(\beta-\overline \beta)(x_1,x_2, x_5)- \psi(( s-\bar s)x_3, ( s- \bar s)x_4)\overline \beta(x_1,x_2, x_5)\\
 &+ \psi((s-\bar s)x_3, sx_4)\overline \beta(x_1,x_2, x_5)+ \psi(sx_3, (s-\bar s)x_4)\overline \beta(x_1,x_2, x_5)\\
 &- \psi(sx_3, sx_5)(\beta-\overline \beta)(x_1,x_2, x_4)+ \psi(( s-\bar s)x_3, ( s- \bar s)x_5)\overline \beta(x_1,x_2, x_4)\\
 &- \psi((s-\bar s)x_3, sx_5)\overline \beta(x_1,x_2, x_4)- \psi(sx_3, (s-\bar s)x_5)\overline \beta(x_1,x_2, x_4)\\
    &+ \psi(sx_4, sx_5)(\beta-\overline \beta)(x_1,x_2, x_3)- \psi(( s-\bar s)x_4, ( s- \bar s)x_5)\overline \beta(x_1,x_2, x_3)\\
 &+ \psi((s-\bar s)x_4, sx_5)\overline \beta(x_1,x_2, x_3)+ \psi(sx_4, (s-\bar s)x_5)\overline \beta(x_1,x_2, x_3)\\
   =&(\beta-\overline \beta)([x_1,x_2, x_3], x_4, x_5)+(\beta-\overline \beta)(x_3,[x_1,x_2, x_4], x_5)\\
 &+(\beta-\overline \beta)(x_1, x_2,[ x_3,x_4, x_5])-(\beta-\overline \beta)(x_3,x_4, [x_1, x_2,x_5]) \\
 &+ \psi(sx_1, sx_2)(\beta-\overline \beta)(x_3,x_4, x_5)+ \psi(sx_3, sx_4)(\beta-\overline \beta)(x_1,x_2, x_5)\\
 &- \psi(sx_3, sx_5)(\beta-\overline \beta)(x_1,x_2, x_4)+ \psi(sx_4, sx_5)(\beta-\overline \beta)(x_1,x_2, x_3)\\
  &-[gx_1, gx_2, [\bar sx_3,\bar sx_4,\bar sx_5]-\bar s[x_3,x_4,x_5]]\\
 &+\psi( sx_2,[\bar sx_3,\bar sx_4, \bar sx_5]-\bar s[x_3,x_4, x_5]])gx_1\\
 &-\psi(sx_1,  [\bar sx_3,\bar sx_4, \bar sx_5]-\bar s[x_3,x_4, x_5]])gx_2\\
 &-[gx_3, gx_4, [\bar sx_1,\bar sx_2, \bar sx_5]-\bar s[x_1,x_2, x_5]]\\
 &+\psi( sx_4, [\bar sx_1,\bar sx_2, \bar sx_5]-\bar s[x_1,x_2, x_5]])gx_3\\
 &+ \psi(sx_3,[\bar sx_1,\bar sx_2,\bar  sx_5]-\bar s[x_1,x_2, x_5]])gx_4\\
  &+ [gx_3, gx_5,[\bar sx_1,\bar sx_2, \bar sx_4]-\bar s[x_1,x_2, x_4]]\\
 &- \psi(sx_5,[\bar sx_1,\bar sx_2,\bar  sx_4]-\bar s[x_1,x_2, x_4]])gx_3\\
 &+\psi(sx_3,[\bar sx_1,\bar sx_2,\bar  sx_4]-\bar s[x_1,x_2, x_4]])gx_5\\
 &- [gx_4, gx_5, [\bar sx_1,\bar sx_2, \bar sx_3]-\bar s[x_1,x_2, x_3]]\\
     &+ \psi( sx_5,[\bar sx_1,\bar sx_2, \bar sx_3]-\bar s[x_1,x_2, x_3]])gx_4\\
 &-\psi(sx_4,[\bar sx_1,\bar sx_2, \bar sx_3]-\bar s[x_1,x_2, x_3]])gx_5.
\end{align*}

Now, define $\tilde\beta : \wedge^3\mathfrak p\mathfrak \to M$ by
\begin{align*}
    \tilde\beta(x_1,x_2,x_3)=&g[x_1,x_2,x_3]-[gx_1,gx_2,gx_3] - \psi( sx_1, \bar sx_3) gx_2+\psi( sx_1,  sx_2)gx_3+\psi(\bar sx_2,  sx_3) gx_1.
\end{align*}
Furthermore, we have
\begin{align*}
   \partial \tilde\beta(x_1,x_2,x_3)=&\partial g[x_1,x_2,x_3]-\partial[gx_1,gx_2,gx_3] - \partial\psi( sx_1, \bar sx_3) gx_2\partial\psi( sx_1,  sx_2)gx_3+\partial\psi(\bar sx_2,  sx_3) gx_1\\
   =&(s-\bar s)[x_1,x_2,x_3]-[(s-\bar s)x_1,(s-\bar s)x_2,(s-\bar s)x_3] - [sx_1, \bar sx_3, (s-\bar s)x_2]\\
   &+[sx_1,  sx_2,(s-\bar s)x_3]+[\bar sx_2,  sx_3, (s-\bar s)x_1]\\
   =&\partial(\beta-\bar \beta)(x_1,x_2,x_3).
\end{align*}
Then $(\beta-\bar \beta -\tilde \beta)(x_1,x_2,x_3)\in \ker \partial=\mathfrak n.$ Thus the map $\beta-\bar \beta -\tilde\beta:\wedge^3\mathfrak p\to\mathfrak n$.

Moreover, if we replace $\beta-\bar \beta$  by $\tilde \beta$ in the expression of $h_\mathcal E-\overline{h}_{\mathcal E}$, then the equality still remains true in $H^3(\mathfrak p;\mathfrak n)$ since
the difference is the coboundary $\delta^3(\beta-\bar \beta -\tilde \beta)$. Using Eqs. \eqref{rep1} and \eqref{rep2} and a routine calculations, we obtain
$$(h_\mathcal E-\overline{h}_{\mathcal E})(x_1,x_2,x_3,x_4,x_5)=\delta^3(\beta-\bar \beta -\tilde \beta)(x_1,x_2,x_3,x_4,x_5).$$
Hence the class of $h_\mathcal E$ does not depend on the section $s$.

Next, consider a map
$$
\xymatrix{
(\mathcal{E}):0 \ar[r] & \mathfrak n \ar[d]^{id_\mathfrak n} \ar[r]^{i} & {M} \ar[d]^{\delta} \ar[r]^{\partial} & {L} \ar[d]^{\gamma} \ar[r]^{\pi} &  {\mathfrak p} \ar[d]^{id_{\mathfrak p}} \ar[r] & 0 \\
(\mathcal{E'}):0 \ar[r] & \mathfrak n  \ar[r]^{i'} &M' \ar[r]^{\partial'} & {L'}  \ar[r]^{\pi'} &  {\mathfrak p}  \ar[r] & 0}
\vspace{.5cm}
$$
of crossed module.
 Let $s': \mathfrak p \to  L'$ and $\sigma': Im(\partial')\to M'$ be sections of $\pi'$ and $\partial'$, respectively. Note
that $(\pi'\gamma s)(x) = (\pi s)(x) = x, ~~\forall x \in \mathfrak p$. Therefore, $\gamma s : \mathfrak p\to L'$ is another section of $\pi'$. Thus, $$\beta'(x_1,x_2,x_3)=\sigma'([\gamma sx_1,\gamma sx_2,\gamma sx_3]-\gamma s[x_1,x_2,x_3]).$$
Using the above notation, we have
 \begin{align*}
     & h_\mathcal E(x_1,x_2, x_3, x_4, x_5)-h_{\mathcal E'}(x_1,x_2, x_3, x_4, x_5)\\
     =&\beta([x_1,x_2, x_3], x_4, x_5)+\beta(x_3,[x_1,x_2, x_4], x_5)+\beta(x_1, x_2,[ x_3,x_4, x_5])\\
      &+\beta(x_3,x_4, [x_1, x_2,x_5]) +\tilde\psi(sx_1, sx_2)\beta(x_3,x_4, x_5)+\tilde\psi(sx_3,sx_4) \beta(x_1, x_2,x_5)\\
      &-\tilde\psi(sx_3,sx_5) \beta(x_1, x_2,x_4)+\tilde\psi(sx_4,sx_5) \beta(x_1, x_2,x_3)\\
      &-      \beta'([x_1,x_2, x_3], x_4, x_5)+\beta'(x_3,[x_1,x_2, x_4], x_5)-\beta'(x_1, x_2,[ x_3,x_4, x_5])\\
      &-\beta'(x_3,x_4, [x_1, x_2,x_5]) -\tilde\psi(\gamma sx_1, \gamma sx_2)\beta'(x_3,x_4, x_5)-\tilde\psi(\gamma sx_3,\gamma sx_4) \beta'(x_1, x_2,x_5)\\
      &+\tilde\psi(\gamma sx_3,\gamma sx_5) \beta'(x_1, x_2,x_4)-\tilde\psi(\gamma sx_4,\gamma sx_5) \beta'(x_1, x_2,x_3)\\
      =&
      (\delta\sigma-\sigma'\gamma)([s[x_1,x_2, x_3], sx_4, sx_5]-s[[x_1,x_2, x_3], x_4, x_5])\\
      &+(\delta\sigma-\sigma'\gamma)([sx_3,s[x_1,x_2, x_4], sx_5]-s[x_3,[x_1,x_2, x_4], x_5])\\
      &+(\delta\sigma-\sigma'\gamma)([sx_1, sx_2,s[x_3,x_4, x_5]]-s[x_1, x_2,[ x_3,x_4, x_5]])\\
      &+(\delta\sigma-\sigma'\gamma)([sx_3,sx_4, s[x_1, x_2,x_5]]-s[x_3,x_4, [x_1, x_2,x_5]]) \\
      &+\tilde\psi(\gamma sx_1, \gamma sx_2)(\delta\sigma-\sigma'\gamma)([sx_3,sx_4, sx_5]-s[x_3,x_4, x_5])\\
      &+\tilde\psi(\gamma sx_3,\gamma sx_4) (\delta\sigma-\sigma'\gamma)([sx_1, sx_2,sx_5]-s[x_1, x_2,x_5])\\
      &-\tilde\psi(\gamma sx_3,\gamma sx_5) (\delta\sigma-\sigma'\gamma)([sx_1, sx_2,sx_4]-s[x_1, x_2,x_4])\\
      &+\tilde\psi(\gamma sx_4,\gamma sx_5) (\delta\sigma-\sigma'\gamma)([sx_1, sx_2,sx_3]-s[x_1, x_2,x_3]).
  \end{align*}
  It follows from the above expression that
  $$h_\mathcal E(x_1,x_2, x_3, x_4, x_5)-h_{\mathcal E'}(x_1,x_2, x_3, x_4, x_5)= \delta^3(\delta\sigma-\sigma'\gamma)(x_1,x_2, x_3, x_4, x_5),$$
  where  $\delta\sigma-\sigma'\gamma:\wedge^3\mathfrak p\to M.$
 Then, $[h_\mathcal E]=[h_{\mathcal E'}]$ in $H^3(\mathfrak p,\mathfrak n)$, where $[h_\mathcal E]$ denoted the equivalence class of $h_\mathcal E$. Therefore, the map $\Phi : Cr_{mod}(\mathfrak n,\mathfrak p)\to H^3(\mathfrak n,\mathfrak p)$ is well-defined.
  \end{proof}

  \begin{rmk}
   We expect an isomorphism
between $Cr_{mod}(\mathfrak n,\mathfrak p)$
and $H^3(\mathfrak n,\mathfrak p)$.  Nevertheless, 
it is not easy to construct a canonical example of crossed extension for a given cohomology class.   
We thus leave it as a conjecture that a such isomorphism exists.
  \end{rmk}

\end{document}